\documentclass[12pt]{article}
\usepackage[dvips]{graphicx}
\usepackage{amsmath,amssymb,amsthm, color}
\def\NZQ{\mathbb}               

\def\ZZ{{\NZQ Z}}
\def\RR{{\NZQ R}}

\def\frk{\mathfrak}               

\def\Phi{{\frk N}}
\def\eb{{\bold e}}

\def\opn#1#2{\def#1{\operatorname{#2}}} 
\opn\chara{char} 
\opn\length{\ell} 
\opn\pd{pd} 
\opn\rk{rk}
\opn\projdim{proj\,dim} 
\opn\injdim{inj\,dim} 
\opn\rank{rank}
\opn\depth{depth} 
\opn\grade{grade} 
\opn\height{height}
\opn\embdim{emb\,dim} 
\opn\codim{codim}

\opn\Tr{Tr} 
\opn\bigrank{big\,rank}
\opn\superheight{superheight}
\opn\lcm{lcm}
\opn\trdeg{tr\,deg}
\opn\reg{reg} 
\opn\lreg{lreg} 
\opn\ini{in} 
\opn\lpd{lpd}
\opn\size{size}
\opn\mult{mult}
\opn\dist{dist}
\opn\cone{cone}
\opn\lex{lex}
\opn\rev{rev}
\opn\div{div} \opn\Div{Div} \opn\cl{cl} \opn\Cl{Cl}
\opn\Spec{Spec} \opn\Supp{Supp} \opn\supp{supp} \opn\Sing{Sing}
\opn\Ass{Ass} \opn\Min{Min}
\opn\Ann{Ann} \opn\Rad{Rad} \opn\Soc{Soc}
\opn\Syz{Syz} \opn\Im{Im} \opn\Ker{Ker} \opn\Coker{Coker}
\opn\Am{Am} \opn\Hom{Hom} \opn\Tor{Tor} \opn\Ext{Ext}
\opn\End{End} \opn\Aut{Aut} \opn\id{id} \opn\ini{in}

\opn\nat{nat}
\opn\pff{pf}
\opn\Pf{Pf} \opn\GL{GL} \opn\SL{SL} \opn\mod{mod} \opn\ord{ord}
\opn\Gin{Gin}
\opn\Hilb{Hilb}\opn\adeg{adeg}\opn\std{std}\opn\ip{infpt}
\opn\Pol{Pol}
\opn\sat{sat}
\opn\Var{Var}
\opn\Gen{Gen}

\opn\aff{aff} \opn\con{conv} \opn\relint{relint} \opn\st{st}
\opn\lk{lk} \opn\cn{cn} \opn\core{core} \opn\vol{vol}
\opn\link{link} \opn\star{star}
\opn\gr{gr}

\def\Gc{{\mathcal G}}

\def\pot#1#2{#1[\kern-0.28ex[#2]\kern-0.28ex]}

\opn\dirlim{\underrightarrow{\lim}}
\opn\inivlim{\underleftarrow{\lim}}
\let\to=\rightarrow

\def\Implies{\ifmmode\Longrightarrow \else
        \unskip${}\Longrightarrow{}$\ignorespaces\fi}
\def\implies{\ifmmode\Rightarrow \else
        \unskip${}\Rightarrow{}$\ignorespaces\fi}
\def\iff{\ifmmode\Longleftrightarrow \else
        \unskip${}\Longleftrightarrow{}$\ignorespaces\fi}

\let\:=\colon
\newtheorem{Theorem}{Theorem}[section]

\newtheorem{Corollary}[Theorem]{Corollary}

\theoremstyle{definition}

\newtheorem{Example}[Theorem]{Example}

\let\epsilon\varepsilon
\let\phi=\varphi
\let\kappa=\varkappa
\def\qed{\ifhmode\textqed\fi
      \ifmmode\ifinner\quad\qedsymbol\else\dispqed\fi\fi}
\def\textqed{\unskip\nobreak\penalty50
       \hskip2em\hbox{}\nobreak\hfil\qedsymbol
       \parfillskip=0pt \finalhyphendemerits=0}
\def\dispqed{\rlap{\qquad\qedsymbol}}

\def\By{{\bf y}}

\def\Bb{{\bf b}}
\def\supp{{\rm supp}}

\def\Bd{{\bf d}}

\def\Bbeta{{\bf \beta}}

\def\Bone{{\bf 1}}

\title{
Markov chain Monte Carlo methods for the\\ 
Box-Behnken designs and\\ centrally symmetric configurations
}
\author{Satoshi Aoki%
\thanks{Graduate School of Science and Engineering (Science Course), Kagoshima University.}
\ 
, Takayuki Hibi%
\thanks{Department of Pure and Applied Mathematics, Graduate School of Information Science and Technology, Osaka University}
and Hidefumi Ohsugi%
\thanks{Department of Mathematical Sciences, School of Science and
Technology,
Kwansei Gakuin University}
}

\begin{document}
\maketitle

\begin{abstract}
We consider Markov chain Monte Carlo methods for calculating
conditional $p$ values of statistical models for count data 
arising in Box-Behnken designs. The statistical model we consider
is a discrete version of the first-order model in the response surface
methodology. For our models, the Markov basis, a key notion to construct a
connected Markov chain on a given sample space, is characterized as generators
of the toric ideals for the centrally symmetric configurations of root
system $D_n$. We show the structure of the Gr\"obner bases for these cases. 
A numerical example for an imaginary data set is given.
\end{abstract}

\section{Introduction}
\label{sec:intro}
After the work by Diaconis and Sturmfels 
(\cite{Diaconis-Sturmfels-1998}), a {\it Markov basis},
a key notion in the field of computational algebraic statistics,
 has attracted special attentions among researchers both in
statistics and algebra.
In this first work, they show the fundamental relation between 
the generators of toric ideals and the Markov bases
and establish
a procedure for sampling from discrete conditional distributions by 
constructing an irreducible Markov chain on a given sample space. 
By virtue of this relation, 
we can perform Markov chain Monte Carlo methods to estimate
conditional $p$ values for various statistical problems if we can obtain 
the generator of corresponding toric ideals. 
Readers can find various theoretical results on 
structure of Markov bases such as minimality or invariance, and 
Markov bases of important statistical models such as hierarchical models of
multi-dimensional contingency tables in \cite{Aoki-Hara-Takemura-2012}.

In parallel, 
it is also valuable to connect known classes of toric 
ideals to statistical models. Such a motivation yields attractive research
topics from algebraic fields to statistics. 
For example, 
\cite{Aoki-Hibi-Ohsugi-Takemura-2010} shows the relation between 
the generator of the toric ideals
for the  Segre--Veronese configuration and 
the special independence models in the testing problems of group-wise 
selections. This result is further
generalized to a class of configurations called nested configurations in 
\cite{Aoki-Hibi-Ohsugi-Takemura-2008}. As another example, relations 
between regular two-level fractional factorial designs and cut ideals
are shown in \cite{Aoki-Hibi-Ohsugi-2013}. 
The arguments in this paper comes from the same motivation to these works.

In this paper, we consider the statistical models corresponding to the 
algebraic object known as a {\it centrally symmetric configuration of root
system $D_n$.} 
The notion of centrally symmetric configurations
(\cite{Ohsugi-Hibi-2013}) is one of the new attractive topics in algebra
since it yields many ``toric rings''  
that have important algebraic properties (normal and Gorenstein).
See \cite{Ohsugi-Hibi-2013}.
On the other hand, 
Gr\"obner bases of the toric ideal 
arising from the configuration of $D_n$
is studied in \cite{Ohsugi-Hibi-2002}.
In addition, convex polytopes arising from 
the centrally symmetric configuration 
of $D_n$ are studied in \cite{rootpoly}.
In this paper, we show that the centrally symmetric configuration corresponds
to the first-order models for the symmetric designs of experiments for  
multi-level factors. As typical examples of such designs, we consider
Box-Behnken designs in this paper. 
Markov chain Monte Carlo procedure in the framework of design of experiments
is introduced in \cite{Aoki-Takemura-2010} and \cite{Aoki-Takemura-2009}. 
In these works, regular two-level and three-level designs are considered.
However, non-regular designs are difficult to treat in general. In this paper,
we present a new method for analyzing non-regular designs. 

The construction of this paper is as follows. In Section 2, we review the 
Markov chain Monte Carlo methods for design of experiments. In Section 3,
we give a definition of the centrally symmetric configuration and 
present statistical models. We also introduce the Box-Behnken designs and
show that the model matrix of the first-order models for the Box-Behnken
designs corresponds to the centrally symmetric configuration of root
system $D_n$. In Section 4, we give the Gr\"obner bases of the
centrally symmetric configuration of root system $D_n$. 
In Section 5, we give numerical example for an imaginary data set. 
Finally, we give some discussion in Section 6.

\section{Markov chain Monte Carlo methods for design of
experiments}
In this section, we introduce Markov chain Monte Carlo methods for
testing the fitting of the log-linear models for fractional factorial
designs with count observations. We consider the designs with $m$ controllable
factors. For $j = 1,\ldots,m$, write $A_j \in \mathbb{Q}$ as the level of
the $j$-th factor. For example, if $j$-th factor has three levels, it
is common to write $A_j = \{-1,0,1\}$. The full factorial design 
$P \subset \mathbb{Q}^m$ is given as $P = A_1\times\cdots\times A_m$ and
the fractional factorial design $F$ is a subset of $P$. 
Suppose there are $k$ runs (i.e., points) in $F$. For convenience, we
order the points of $F$ appropriately and consider a $k\times m$ design 
matrix $D = (d_{ij})$, where $d_{ij}$ is the level of $j$-th factor in $i$-th
run for $i = 1,\ldots,k,\ j = 1,\ldots,m$. 

We write the observations as $\By = (y_1,\ldots,y_k)'$, where $'$ denotes
the transpose. In this paper, we consider the Poisson sampling scheme, i.e., 
we suppose that the observations are counts of some events and 
the observation $\By$ is nonnegative integer vector.
We also suppose that only one
observation is obtained for each run. This is a natural setting 
because the set of the totals for each run is the sufficient 
statistics for the parameter in the Poisson sampling scheme. 
Therefore the observation $\By$ are realizations from $k$ 
mutually independent Poisson
random variables $Y_1,\ldots,Y_k$ with the mean 
parameter $\lambda_i = E(Y_i),\ i = 1,\ldots,k$.

We consider the log-linear model 
\begin{equation}
\log \lambda_i = \beta_0 + \beta_1x_{i1} + \cdots + \beta_{n}x_{in}
,\ i = 1,\ldots,k
\label{eqn:log-linear-model}
\end{equation}
for the parameter $\lambda_i, i = 1,\ldots,k$, where $x_{ij}$ is a $j$-th 
covariate for the $i$-th run and $n+1$ is the dimension 
of the parameter 
$\Bbeta = (\beta_0, \beta_1,\ldots,\beta_{n})'$.
If we write $x_{i0} = 1$ for $i = 1,\ldots,k$, the log-linear model
(\ref{eqn:log-linear-model}) is written as
\[
(\log \lambda_1,\ldots,\log\lambda_k)' = M\Bbeta,
\]
where $M = (x_{ij})_{i = 1,\ldots,k; j = 0,\ldots,n}$.
We call a $k \times (n+1)$ matrix $M$ as a {\it model matrix}
 of the log-linear model
(\ref{eqn:log-linear-model}).

To judge the fitting of the log-linear model 
(\ref{eqn:log-linear-model}), we can consider various goodness-of-fit tests.
In the goodness-of-fit tests, the model 
(\ref{eqn:log-linear-model}) is treated as the null model, whereas the
saturated model is treated as the alternative model. Under the null model
(\ref{eqn:log-linear-model}), the sufficient statistics for the (nuisance) 
parameter $\Bbeta$ is given by $M'\By$ from the factorization
\[
\prod_{i=1}^{k}e^{-\lambda_i}\frac{\lambda_i^{y_i}}{y_i!}
= \left(\prod_{i=1}^k\frac{1}{y_i!}\right)
\exp\left(\Bbeta'M'\By - \sum_{i=1}^k\lambda_i\right).
\]
Therefore the conditional distribution of $\By$ for the given sufficient 
statistics is written as
\begin{equation}
f(\By\ |\ M'\By = M'\By^o) = C(M'\By^o)^{-1}\prod_{i=1}^k\frac{1}{y_i!},
\label{eqn:conditional-distribution}
\end{equation}
where $\By^o$ is the observation vector and $C(M'\By^o)$ is the normalizing
constant determined from the sufficient statistics $M'\By^o$ as
\begin{equation}
C(M'\By^o) = \sum_{\By \in {\cal F}(M'\By^o)}\left(
\prod_{i=1}^k\frac{1}{y_i!}
\right)
\label{eqn:normalizing-constant}
\end{equation}
and
\begin{equation}
{\cal F}(M'\By^o) = \{\By\in \mathbb{Z}_{\geq 0}^k\ |\ M'\By = M'\By^o\}.
\label{eqn:fiber}
\end{equation}

In this paper, we consider goodness-of-fit tests based on the
conditional distribution (\ref{eqn:conditional-distribution}).
There are several choices of the test statistics $T(\By)$. Frequently
used choices are the likelihood ratio statistics
\begin{equation}
T(\By) = 2\sum_{i=1}^ky_i\log\frac{y_i}{\hat{\lambda}_i}
\label{eqn:likelihood-ratio}
\end{equation}
or the Pearson $\chi^2$ statistics
\[
T(\By) = \sum_{i=1}^{k}\frac{(y_i - \hat{\lambda}_i)^2}{\hat{\lambda}_i},
\]
where $\hat{\lambda}_i$ is the maximum likelihood estimate for $\lambda_i$
under the null model (i.e., fitted value). A simple way of judging the 
significance for the observed value $T(\By^o)$ is the asymptotic $p$ value
based on the asymptotic distribution $\chi_{k - n - 1}^2$ of the test
statistics. 
However, the fitting of the asymptotic approximation may be sometimes poor.
Therefore we consider conditional exact $p$ values in this paper. 
Based on the conditional distribution (\ref{eqn:conditional-distribution}), 
the exact conditional $p$ value is written as
\begin{equation}
p = \sum_{\By \in {\cal F}(M'\By^o)}f(\By\ |\ M'\By = M'\By^o)\Bone(T(\By)
\geq T(\By^o)),
\label{eqn:exact-p-value}
\end{equation}
where
\begin{equation}
\Bone(T(\By) \geq T(\By^o)) = \left\{\begin{array}{ll}
1, & \mbox{if}\ T(\By) \geq T(\By^o),\\
0, & \mbox{otherwise}
\end{array}
\right.
\label{eqn:test-function}
\end{equation}
is the test function of $T(\By)$. 
Of course, if we can calculate the exact $p$ value of (\ref{eqn:exact-p-value})
and (\ref{eqn:test-function}), it is best. However, 
the cardinality of the set ${\cal F}(M'\By^o)$ becomes huge for 
moderate
sizes of data and the calculation of the normalizing constant $C(M'\By^o)$ 
of (\ref{eqn:normalizing-constant}) is usually computationally infeasible. 
Instead, we consider a Markov chain 
Monte Carlo method to evaluate the conditional $p$ values. It should be
noted that we need not calculate the normalizing constant 
(\ref{eqn:normalizing-constant}) to evaluate the $p$ values by the Markov
chain Monte Carlo methods. This point is one of the important advantages
of the Markov chain Monte Carlo methods. 

To perform the Markov chain Monte Carlo procedure, we have to construct
an irreducible Markov chain over the conditional sample space 
(\ref{eqn:fiber}) with the stationary distribution 
(\ref{eqn:conditional-distribution}). If such a chain is constructed, 
we can sample from the chain as $\By^{(1)},.\ldots,\By^{(T)}$ after
discarding some initial burn-in steps, and estimate the $p$ values as
\[
\hat{p} = \frac{1}{T}\sum_{t=1}^T \Bone(T(\By^{(t)} \geq T(\By^o)).
\]
Such a chain can be constructed easily by {\it Markov bases}. 
Once a Markov basis is obtained, we can construct a connected, aperiodic and
reversible Markov chain over the conditional sample space 
(\ref{eqn:fiber}), which can be modified so as to have the stationary 
distribution (\ref{eqn:conditional-distribution}) by the Metropolis-Hastings
procedure. See \cite{Hastings-1970} or \cite{Diaconis-Sturmfels-1998} for
details.

The Markov basis is characterized algebraically as follows. 
Write the variables $x_1,\ldots,x_k$ and consider the polynomial
ring $K[x_1,\ldots,x_k]$ over a field $K$. Consider the integer kernel
of the transpose of the model matrix $M$, ${\rm Ker}_{\mathbb{Z}}M'$. 
For each $\Bb = (b_1,\ldots,b_k)' \in {\rm Ker}_{\mathbb{Z}}M'$, define
a binomial in $K[x_1,\ldots,x_k]$ as
\[
f_{\Bb} = \prod_{b_i > 0}x_i^{b_i} - \prod_{b_i < 0}x_i^{-b_i}\ .
\]
Then the binomial ideal in $K[x_1,\ldots,x_k]$, 
\[
I_{M'} = \left<\{f_{\Bb}\ |\ \Bb \in {\rm Ker}_{\mathbb{Z}}M'\}\right>
\]
is called a {\it toric ideal of the configuration $M'$}.
Then for a generating set of $I_{M'}$, 
$\{f_{\Bb^{(1)}},\ldots,f_{\Bb^{(s)}}\}$, the set of integer vectors
$\{\Bb^{(1)},\ldots,\Bb^{(s)}\}$ constitutes a Markov basis. 
See \cite{Diaconis-Sturmfels-1998} for details.

\section{Statistical models of the centrally symmetric 
configurations and Box-Behnken designs}
As we have seen in Section 2, if we can obtain a generator of $I_{M'}$, 
a toric ideal of the configuration $M'$, we can judge the fitting of
the statistical model expressed by the model matrix $M$ by the conditional
$p$ values estimated by the Markov chain Monte Carlo methods. For small sizes
of problems, we can rely on various softwares such as 
4ti2 (\cite{4ti2}) to compute generators of the toric ideals. However, 
for problems of large sizes, it is usually very difficult to compute
Markov bases or Gr\"obner bases for given ideals. 
On the other hand, if we 
have theoretical results on the structure of the corresponding ideals,
it is very easy to perform the Markov chain Monte Carlo procedure for
such configurations.
The {\it centrally symmetric configuration} is an example of such cases.

The centrally symmetric configuration is given in (\cite{Ohsugi-Hibi-2013}) 
as follows. Let $A \in \mathbb{Z}^{n\times s}$ be an integer matrix for which
no column vector is a zero vector. Then the centrally symmetric configuration
of $A$ is the $(n+1)\times (2s+1)$ integer matrix
\begin{equation}
A^{\pm} = \left(\begin{array}{c|ccc|ccc}
0 & & & & & & \\
\vdots & & A & & & -A & \\
0 & & & & & & \\ \hline
1 & 1 & \cdots & 1 & 1 & \cdots & 1
\end{array}
\right)\ .
\label{eqn:centrally-symmetric}
\end{equation}
It is known that
the ``toric ring''  $K[A^\pm]$ of $A^\pm$ is normal
and  Gorenstein
if there exists a squarefree initial ideal with respect to
a reverse lexicographic order where the smallest variable
corresponds to the first column of $A^\pm$.
See, e.g., \cite[Lemma 1.1]{HMOS-2014}.

As natural statistical models and designs where (the transpose of) the
 model matrix is centrally symmetric configurations, 
we consider the simple first-order models as follows. Suppose 
$F \subset A_1\times \cdots A_m \in \mathbb{Q}^m$ 
is a {\it symmetric design including the origin}, i.e., a design satisfying
\[
(0,\ldots,0) \in F\ \ \mbox{and}\ \ \Bd \in F\ \Rightarrow\ -\Bd \in F.
\]
Write $D = (d_{ij}) \in \mathbb{Z}^{k\times m}$ its design matrix, 
where $d_{ij}$ is the level of $j$-th factor in $i$-th run for
$i = 1,\ldots,k,\ j=1,\ldots,m$. Then we see that the transpose of the model
matrix
\begin{equation}
M = \left(\begin{array}{cc}
1 & \\
\vdots & D\\
1 & 
\end{array}
\right)
\label{eqn:M-for-first-order-model}
\end{equation}
is centrally symmetric. Corresponding log-linear model 
(\ref{eqn:log-linear-model})
is written as
\begin{equation}
\log \lambda_i = \beta_0 + \beta_1 d_{i1} + \cdots + \beta_m d_{im},\ 
i = 1,\ldots,k. 
\label{eqn:first-order-model}
\end{equation}
We call the model (\ref{eqn:first-order-model}) 
as a {\it first-order model} in this paper. 
The interpretation of the first-order model (\ref{eqn:first-order-model}) is 
as follows. Suppose there are adequate meanings both in 
the order of the factors 
and the interval of the factors for the design $F$. Then the model 
(\ref{eqn:first-order-model}) means that the logarithm of the influence
to the response variable is
proportional to the difference of the levels for each factor. 
In other words, the parameter $\beta_j$ represents the main effect of the
$j$-th factor for $j = 1,\ldots,m$. 
The first order model (\ref{eqn:first-order-model}) is a discrete version 
of the first-order model arising in the context of the 
{\it response surface methodology}. See Section 9 of \cite{Wu-Hamada}, for
example. A typical example of the symmetric designs is also arising in the
context of the response surface methodology as {\it Box-Behnken designs}. 

The Box-Behnken design is a family of three-level fractional factorial 
designs introduced by \cite{Box-Behnken-1960}. This design is constructed
by combining two-level factorial designs with balanced (or partially balanced) 
incomplete block designs
in a particular manner. To illustrate the concept of the Box-Behnken designs, 
consider the case of three factors (i.e., $m = 3$). A balanced incomplete
block design with three factors and three blocks is given as follows.
\[
\begin{array}{c|ccc}\hline
& \multicolumn{3}{|c}{\mbox{Factor}}\\
\mbox{Block} &\  1\  & 2 &\ 3\ \\ \hline
1 & \circ & \circ & \\
2 & \circ & & \circ\\
3 & & \circ & \circ\\ \hline
\end{array}
\]
The Box-Behnken design is constructed by replacing the two circles ($\circ$)
in each block by the two columns of the two-level $2^2$ design and add a column
of zeros where a circle does not appear, and adding a run at the origin. 
In this example, the Box-Behnken design is constructed as follows.
\[
\begin{array}{rrr}\hline
\multicolumn{3}{c}{\mbox{Factor}}\\
\multicolumn{1}{c}{1} & \multicolumn{1}{c}{2} & \multicolumn{1}{c}{3} \\ \hline
-1 & -1 & 0\\
-1 & 1 & 0\\
1 & -1 & 0\\
1 & 1 & 0\\
-1 & 0 & -1\\
-1 & 0 & 1\\
1 & 0 & -1\\
1 & 0 & 1\\
0 & -1 & -1\\
0 & -1 & 1\\
0 & 1 & -1\\
0 & 1 & 1\\
0 & 0 & 0\\ \hline
\end{array}
\]
Similarly, by combining various incomplete block designs with
two-level full (or fractional) factorial designs, various three-level
fractional factorial designs are obtained. 
In this paper, we only consider the Box-Behnken designs constructed from
the two-level $2^2$ design and the
balanced incomplete block designs with the block size $2$, 
the number of factors (or treatments) $m$, 
the number of blocks $m(m-1)/2$ and 
the number of replicates for each factor $m-1$, 
with a single run at the origin. Note that
it is common to consider the designs with several
runs at the origins in this field. 
See \cite{Box-Behnken-1960} or Chapter 9 of \cite{Wu-Hamada} for details.
However, we only consider the designs
with single observations even in the origin. 
Therefore the Box-Behnken design considered in this paper has $2m(m-1) + 1$
runs for $m$ factor case. 

For these Box-Behnken designs, we consider the first-order model 
(\ref{eqn:first-order-model}) with the model matrix 
(\ref{eqn:M-for-first-order-model}). Note that $n = m$ in our cases. 
Then we see that the transpose of the model matrix, $M'$, has the
centrally symmetric structure (\ref{eqn:centrally-symmetric}) 
with $s = m(m-1)$. For example, the transpose of the 
model matrix of the first-order model
for the three factors case is given by
\[
\left(
\begin{array}{rrrrrrrrrrrrr}
0 & -1 & -1 & 1 & 1 & -1 & -1 & 1 & 1 & 0 & 0 & 0 & 0 \\
0 & -1 &  1 &-1 & 1 &  0 &  0 & 0 & 0 &-1 &-1 & 1 & 1\\
0 &  0 &  0 & 0 & 0 & -1 &  1 &-1 & 1 &-1 & 1 &-1 & 1\\ 
1 &  1 &  1 & 1 & 1 &  1 &  1 & 1 & 1 & 1 & 1 & 1 & 1
\end{array} 
\right),
\]
which is the centrally symmetric configuration of
\[
\left(
\begin{array}{rrrrrr}
-1 & -1 & -1 & -1 & 0 & 0\\
-1 &  1 &  0 &  0 &-1 &-1\\
 0 &  0 & -1 &  1 &-1 & 1
\end{array}
\right).
\]

Following the arguments of Section 2, we can judge the fitting of the 
first-order model for the Box-Behnken designs by the Markov chain Monte 
Carlo methods, if we obtain the generators of the toric ideal of 
the configuration of this type.

\section{Gr\"obner bases of centrally symmetric configurations
 of root system $D_n$}
Now we show the structure of the Gr\"obner bases of the centrally 
symmetric configurations for the Box-Behnken designs. 
Because the Gr\"obner basis is a generator of the ideals, we can use the
Gr\"obner basis as a Markov basis. 
As an important fact, the transpose of the model matrix for 
the first-order models for
the Box-Behnken designs is characterized as the configuration of the 
{\it root system $D_n$}. 

Let $\eb_{1}, \ldots, \eb_{n}$ stand for the canonical unit coordinate
vectors of $\RR^{n}$ and ${\bf D}_{n}^{\pm} \subset \RR^{n}$ 
the finite set which consists of the origin ${\bf 0}$ of $\RR^{n}$ 
together with
\[
\eb_{i}+\eb_{j}, \, \, \eb_{i}-\eb_{j}, \, \, -\eb_{i}+\eb_{j}, 
\, \, -\eb_{i}-\eb_{j}, \, \, \, \, \, 1 \leq i < j \leq n.
\]
Let $K[{\bf t}, {\bf t}^{-1}, s] 
= K[t_{1}, \ldots, t_{n}, t_{1}^{-1}, \ldots, t_{n}^{-1}, s]$
denote the Laurent polynomial ring in $2n + 1$ variables over a field $K$.
The {\em toric ring} of ${\bf D}_{n}^{\pm}$ is the subring 
$K[{\bf D}_{n}^{\pm}]$ of 
$K[{\bf t}, {\bf t}^{-1}, s]$
which is generated by $s$ together with 
$t_{i}t_{j}s, t_{i}t_{j}^{-1}s, t_{i}^{-1}t_{j}s, t_{i}^{-1}t_{j}^{-1}s$, 
where $1 \leq i < j \leq n$.
Let $K[{\bf x}, z]$ be the polynomial ring over $K$
in the variables $z$ together with $x_{ij}^{pq}$, where
$1 \leq i < j \leq n$ and $p, q \in \{ +, - \}$.
We then define the surjective ring homomorphism 
$\pi \, : \, K[{\bf x}, z] \to K[{\bf D}_{n}^{\pm}]$ by setting
$\pi(z) = s$ and $\pi(x_{ij}^{pq}) = t_{i}^{p'}t_{j}^{q'}s$,
where $+' = 1$ and $-' = -1$.  For example
$\pi(x_{25}^{-+}) = t_{2}^{-1}t_{5}s$.
The {\em toric ideal} of ${\bf D}_{n}^{\pm}$ is the kernel
$I_{{\bf D}_{n}^{\pm}}$ of $\pi$.

Fix an ordering $<$ of the variables of $K[{\bf x}]$ with the property that
$x_{ij}^{pq} < x_{k\ell}^{rs}$ if either (i) $i < k$ 
or (ii) $i = k$ and $j > \ell$.  Let $<_{\rm lex}$ denote the lexicographic order
on $K[{\bf x}]$ induced by the ordering $<$.  We introduce the monomial
order $\prec$ on $K[{\bf x}, z]$ defined as follows:
One has
$\prod_{\xi=1}^{a}x_{i_{\xi}j_{\xi}}^{p_{\xi}q_{\xi}}z^{\alpha} 
\prec 
\prod_{\nu=1}^{b}x_{k_{\nu}\ell_{\nu}}^{r_{\nu}s_{\nu}}z^{\beta}$,
where $\alpha, \beta \in \ZZ_{\geq 0}$, 
if 
\begin{itemize}
\item
$a + \alpha < b + \beta$, or
\item
$a + \alpha = b + \beta$ and $\alpha > \beta$, or
\item
$a = b, \alpha = \beta$ and
$\prod_{\xi=1}^{a}x_{i_{\xi}j_{\xi}}^{p_{\xi}q_{\xi}} 
<_{\rm lex}
\prod_{\nu=1}^{b}x_{k_{\nu}\ell_{\nu}}^{r_{\nu}s_{\nu}}$. 
\end{itemize}

Let $\Gc$ denote the set of binomials
\begin{enumerate}
\item[(i)]$x_{ij}^{pq}x_{k\ell}^{rs} - x_{ik}^{pr}x_{j\ell}^{qs}$, \, 
$1 \leq i < j < k < \ell \leq n$;
\item[(ii)]$x_{i\ell}^{ps}x_{jk}^{qr} - x_{ik}^{pr}x_{j\ell}^{qs}$, \, 
$1 \leq i < j < k < \ell \leq n$;
\item[(iii)]$x_{ij}^{+p}x_{ik}^{-q} - x_{jk}^{pq}z$, \, 
$|\{i,j,k\}|=3$;\footnote{\,
For $j < i$, 
the notation $x_{ij}^{pq}$ is identified with
the variable $x_{ji}^{qp}$.}
\item[(iv)]$x_{ij}^{++}x_{ij}^{--} - z^{2}$, \, $1 \leq i < j \leq n$;
\item[(v)]$x_{ij}^{+-}x_{ij}^{-+} - z^{2}$, \, $1 \leq i < j \leq n$;
\item[(vi)]$x_{ij}^{p+}x_{ij}^{p-} - x_{1i}^{+p}x_{1i}^{-p}$, \,
$1 < i \neq j \leq n$;
\item[(vii)]$x_{1j}^{p+}x_{1j}^{p-} - x_{1n}^{p+}x_{1n}^{p-}$, \,
$1 < j < n$;
\item[(viii)]$x_{1i}^{+p}x_{1i}^{-p}x_{jk}^{qr} - x_{ij}^{pq}x_{ik}^{pr}z$, \,
$1 < i, j, k \leq n, |\{i,j,k\}|=3$
\end{enumerate} 
belonging to $I_{{\bf D}_{n}^{\pm}}$.

\begin{Theorem}
\label{gb}
The set $\Gc$ of binomials is a Gr\"obner basis of
$I_{{\bf D}_{n}^{\pm}}$ with respect to $\prec$.
\end{Theorem}

\begin{proof}
In general, if $f = u - v$ is a binomial, then $u$ is called the {\em first}
monomial of $f$ and $v$ is called the {\em second} monomial of $f$. 
The initial monomial of each of the binomials (i) -- (viii) 
with respect to $\prec$ is its first monomial. 
Let ${\rm in}_{\prec}(\Gc)$ denote the set of initial monomials of binomials 
belonging to $\Gc$.  It follows from \cite[(0.1)]{Ohsugi-Hibi-2002} that,
in order to show that $\Gc$ is a Gr\"obner basis of
$I_{{\bf D}_{n}^{\pm}}$ with respect to $\prec$, what we must prove is
the following:
($\clubsuit$) If $u$ and $v$ are monomials belonging to $K[{\bf x}, z]$
with $u \neq v$ such that 
$u \not\in {\rm in}_{\prec}(\Gc)$ and $v \not\in {\rm in}_{\prec}(\Gc)$,
then $\pi(u) \neq \pi(v)$. 

Let $u$ and $v$ be monomials belonging to $K[{\bf x}, z]$.
Write
\[
u = x_{i_{1}j_{1}}^{p_{1}q_{1}} \cdots x_{i_{a}j_{a}}^{p_{a}q_{a}}
z^{\alpha}, 
\, \, \, \, \, \, \, \, \, \, 
v = x_{k_{1}\ell_{1}}^{r_{1}s_{1}} \cdots x_{k_{b}\ell_{b}}^{r_{b}s_{b}}
z^{\beta}
\]
with
\[
i_{1} \leq \cdots \leq i_{a}, \, \, \, \, \, \, \, \, \, \, 
k_{1} \leq \cdots \leq k_{b}.
\]

Let $\pi(u) = \pi(v)$.
Then $a + \alpha = b + \beta$.
Suppose that 
$u \not\in {\rm in}_{\prec}(\Gc)$ and $v \not\in {\rm in}_{\prec}(\Gc)$.
Furthermore, suppose that $u$ and $v$ are relatively prime.
Especially either $\alpha = 0$ or $\beta = 0$.  
Let, say, $\alpha = 0$.
In other words,
\[
u = x_{i_{1}j_{1}}^{p_{1}q_{1}} \cdots x_{i_{a}j_{a}}^{p_{a}q_{a}}, 
\, \, \, \, \, \, \, \, \, \, 
v = x_{k_{1}\ell_{1}}^{r_{1}s_{1}} \cdots x_{k_{b}\ell_{b}}^{r_{b}s_{b}}
z^{\beta},
\]
where $\beta = a - b$.

Let $i_{a'} < i_{a''}$, where $1 \leq a' < a'' \leq a$.  
Then, by using (i), one has $i_{a''} \leq j_{a'}$
and, by using (ii), one has $j_{a'} \leq j_{a''}$.  Hence
$i_{a'} < i_{a''} \leq j_{a'} \leq j_{a''}$.
It then follows that
\begin{eqnarray}
\label{Boston}
i_{1} \leq \cdots \leq i_{a} \leq j_{1} \leq \cdots \leq j_{a}, 
\, \, \, \, \, \, \, \, \, \, 
k_{1} \leq \cdots \leq k_{b} \leq \ell_{1} \leq \cdots \leq \ell_{b}.
\end{eqnarray}

We claim that none of the followings arises:
\begin{enumerate}
\item[{\bf (\,$\sharp$\,)}]
$i_{a'} = i_{a''}$ and $p_{a'} \neq p_{a''}$; 
\item[{\bf (\,$\flat$\,)}]
$i_{a} = j_{1}$ and $p_{a} \neq q_{1}$;
\item[{\bf (\,$\natural$\,)}]
$j_{a'} = j_{a''}$ and $q_{a'} \neq q_{a''}$.  
\end{enumerate}

\medskip

\noindent
{\bf (Case (\,$\sharp$\,))}
Let $i_{a'} = i_{a''}$ and $p_{a'} \neq p_{a''}$, where
$1 \leq a' < a'' \leq a$.
Then, by using (iii), one has $j_{a'} = j_{a''}$.
Then, by using (iv) and (v), one has $q_{a'} = q_{a''}$.
Moreover, by using (vi), one has $i_{a'} = i_{a''} = 1$. 
Thus $x_{i_{a'}j_{a'}}^{p_{a'}q_{a'}}x_{i_{a''}j_{a''}}^{p_{a''}q_{a''}}
= x_{1j_{a'}}^{+q_{a'}}x_{1j_{a'}}^{-q_{a'}}$ divides $u$.  
Then, by using (viii),
each variable $x_{i_{a^{*}}j_{a^{*}}}^{p_{a^{*}}q_{a^{*}}}$
($\neq x_{1j_{a'}}^{+q_{a'}}$, $\neq x_{1j_{a'}}^{-q_{a'}}$)
which divides $u$ satisfies either $i_{a^{*}} = 1$
or $|\{j_{a'},i_{a^{*}},j_{a^{*}}\}| = 2$ with $i_{a^{*}} > 1$.
Then, by using (iii), if $i_{a^{*}} = 1$, then $j_{a'} = j_{a^{*}}$.
Thus $q_{a'} \neq q_{a^{*}}$.
However, by using (iv) and (v), a contradiction arises.
Hence either $j_{a'} = i_{a^{*}} > 1$ or
$j_{a'} = j_{a^{*}} > 1$.  In other words, $u$ is divided by 
either $x_{1j_{a'}}^{+q_{a'}}x_{1j_{a'}}^{-q_{a'}}
x_{j_{a'}j_{a^{*}}}^{p_{a^{*}}q_{a^{*}}}$ or
$x_{1j_{a'}}^{+q_{a'}}x_{1j_{a'}}^{-q_{a'}}
x_{i_{a^{*}}j_{a'}}^{p_{a^{*}}q_{a^{*}}}$.
Then, by using (iii), if $j_{a'} = i_{a^{*}}$, then $q_{a'} = p_{a^{*}}$.
Again, by using (iii),  
if $j_{a'} = j_{a^{*}}$, then $q_{a'} = q_{a^{*}}$.
Hence either 
$x_{i_{a^{*}}j_{a^{*}}}^{p_{a^{*}}q_{a*}} 
= x_{j_{a'}j_{a^{*}}}^{q_{a'}q_{a^{*}}}$
or
$x_{i_{a^{*}}j_{a^{*}}}^{p_{a^{*}}q_{a*}} 
= x_{i_{a^{*}}j_{a'}}^{p_{a^{*}}q_{a'}}$.
As a result, either $t_{j_{a'}}^{a}$ or $t_{j_{a'}}^{-a}$
divides $\pi(u)$.  Let, say, $t_{j_{a'}}^{a}$ divides $\pi(u)$.
Since $\pi(u) = \pi(v)$, it follows that
$\beta = a - b = 0$ and that $t_{j_{a'}}^{a}$ divides $\pi(v)$.

Let, say, either $t_{j}$ or $t_{j}^{-1}$, 
where $j \neq j_{a'}$, divide $\pi(u) = \pi(v)$. 
Then either $x_{j_{a'}j}^{++}$ or $x_{j_{a'}j}^{+-}$ must divide
both $u$ and $v$, which contradicts the fact that
$u$ and $v$ are relatively prime.  Hence 
$\pi(u) = \pi(v) = t_{j_{a'}}^{a} s^{a}$. 
Thus a quadratic monomial $x_{j_{a'}j_{0}}^{++}x_{j_{a'}j_{0}}^{+-}$,
where $j_{0} \neq j_{a'}$, divides $u$.  Then, since $j_{a'} > 1$,
by using (vi), one has $j_{0} = 1$.  Hence $x_{1j_{a'}}^{++}x_{1j_{a'}}^{-+}$
divides both $u$ and $v$, a contradiction.

\medskip

\noindent
{\bf (Case (\,$\flat$\,))}
Let $i_{a} = j_{1}$ and $p_{a} \neq q_{1}$.  Then, by using (iii), 
one has $u \in {\rm in}_{\prec}(\Gc)$.

\medskip

\noindent
{\bf (Case (\,$\natural$\,))}
Let $j_{a'} = j_{a''}$ and $q_{a'} \neq q_{a''}$, where  
$1 \leq a' < a'' \leq a$.
Then, by using (iii), one has $i_{a'} = i_{a''}$.
Furthermore, by using (iv) and (v), one has $p_{a'} = p_{a''}$.
If $i_{a'} > 1$, then $u$ is divided by 
$x_{i_{a'}j_{a'}}^{p_{a'}q_{a'}}x_{i_{a''}j_{a''}}^{p_{a''}q_{a''}}
= x_{i_{a'}j_{a'}}^{p_{a'}q_{a'}}x_{i_{a'}j_{a'}}^{p_{a'}-q_{a'}}$.
Thus, by using (vi), one has $u \in {\rm in}_{\prec}(\Gc)$.
Hence $i_{a'} = i_{a''} = 1$.  Thus, by using (vii), 
one has $j_{a'} = j_{a''} = n$.

Let $i_{a^{*}} > 1$ for some $1 \leq a^{*} \leq n$. 
Then $a' < a^{*}$ and $1 = i_{a'} < i_{a^{*}} < n = j_{a'} \leq j_{a^{*}}$.
Hence $j_{a'} = j_{a^{*}} = n$.  
However, since $i_{a^{*}} > 1$, it follows that
$q_{a'} = q_{a^{*}}$.  
Thus 
$x_{i_{a''}j_{a''}}^{p_{a''}q_{a''}}
x_{i_{a^{*}}j_{a^{*}}}^{p_{a^{*}}q_{a^{*}}}
= x_{1n}^{p_{a'}-q_{a'}}
x_{i_{a^{*}}n}^{p_{a^{*}}q_{a'}}$.
Since $i_{a^{*}} < n$, by using (iii), 
one has $u \in {\rm in}_{\prec}(\Gc)$.
As a result, $i_{a^{*}} = 1$ for all $1 \leq a^{*} \leq a$.

Now, since {\bf (\,$\sharp$\,)} cannot occur, one has $p_{1} = p_{a^{*}}$
for all $1 \leq a^{*} \leq a$.  Thus $\pi(u)$ is divided by
either $t_{1}^{a}$ or $t_{1}^{-a}$.  
Let, say, $t_{1}^{a}$ divide $\pi(u)$.
Since $\pi(u) = \pi(v)$,
it follows that $\beta = a - b = 0$ and $t_{1}^{a}$ divides $\pi(v)$.
Let, say, either $t_{j}$ or $t_{j}^{-1}$, 
where $j \neq 1$, divide $\pi(u) = \pi(v)$. 
Then either $x_{1j}^{++}$ or $x_{1j}^{+-}$ must divide
both $u$ and $v$, which contradicts the fact that
$u$ and $v$ are relatively prime.  Hence 
$\pi(u) = \pi(v) = t_{1}^{a} s^{a}$. 
Thus a quadratic monomial $x_{1j_{0}}^{++}x_{1j_{0}}^{+-}$,
where $j_{0} \neq 1$, divides $u$.  Then,
by using (vii), one has $j_{0} = n$.  Thus $x_{1n}^{++}x_{1n}^{+-}$
divides both $u$ and $v$, a contradiction.

\medskip

Finally, since none of {\bf (\,$\sharp$\,)}, {\bf (\,$\flat$\,)}
and {\bf (\,$\natural$\,)} arises, it follows that 
no cancellation occurs in the expression of the Laurent monomial 
\[
\pi(u) = t_{i_{1}}^{p'_{1}} \cdots t_{i_{a}}^{p'_{a}}
t_{j_{1}}^{q'_{1}} \cdots t_{j_{a}}^{q'_{a}}.
\]
Since $\pi(u) = \pi(v)$, one has $\beta = a - b = 0$.
Furthermore, since no cancellation occurs in the expression of 
the Laurent monomial 
\[
\pi(v) = t_{k_{1}}^{r'_{1}} \cdots t_{k_{a}}^{r'_{a}}
t_{\ell_{1}}^{s'_{1}} \cdots t_{\ell_{a}}^{s'_{a}},
\]
it follows from (\ref{Boston}) that
\[
i_{\xi} = k_{\xi}, \, j_{\xi} = \ell_{\xi}, \, p_{\xi} = r_{\xi}, \,
q_{\xi} = s_{\xi}, \, \, \, \, \, 1 \leq \xi \leq a.
\]
Recall that $u$ and $v$ are relatively prime.  Hence $a = 0$ and $u = v = 1$.
Consequently the required condition ($\clubsuit$) is satisfied.
\, \, \, \, \, \, \, \, \, \,
\, \, \, \, \, \, \, \, \, \, 
\, \, \, \, \, \, \, \,
\end{proof}

\begin{Example} \label{ex:n=3-case}
Consider the case of $n=3$. 
From Theorem \ref{gb}, 
\begin{enumerate}

\item[(iii)]
$x_{12}^{++}x_{13}^{-+} - x_{23}^{++}z$, \, 
$x_{12}^{++}x_{13}^{--} - x_{23}^{+-}z$, \, 
$x_{12}^{+-}x_{13}^{-+} - x_{23}^{-+}z$, \, 
$x_{12}^{+-}x_{13}^{--} - x_{23}^{--}z$, \, 

$x_{12}^{-+}x_{13}^{++} - x_{23}^{++}z$, \, 
$x_{12}^{-+}x_{13}^{+-} - x_{23}^{+-}z$, \, 
$x_{12}^{--}x_{13}^{++} - x_{23}^{-+}z$, \, 
$x_{12}^{--}x_{13}^{+-} - x_{23}^{--}z$, \, 

$x_{12}^{++}x_{23}^{-+} - x_{13}^{++}z$, \, 
$x_{12}^{++}x_{23}^{--} - x_{13}^{+-}z$, \, 
$x_{12}^{-+}x_{23}^{-+} - x_{13}^{-+}z$, \, 
$x_{12}^{-+}x_{23}^{--} - x_{13}^{--}z$, \, 

$x_{12}^{+-}x_{23}^{++} - x_{13}^{++}z$, \, 
$x_{12}^{+-}x_{23}^{+-} - x_{13}^{+-}z$, \, 
$x_{12}^{--}x_{23}^{++} - x_{13}^{-+}z$, \, 
$x_{12}^{--}x_{23}^{+-} - x_{13}^{--}z$, \,

$x_{13}^{++}x_{23}^{+-} - x_{12}^{++}z$, \, 
$x_{13}^{++}x_{23}^{--} - x_{12}^{+-}z$, \, 
$x_{13}^{-+}x_{23}^{+-} - x_{12}^{-+}z$, \, 
$x_{13}^{-+}x_{23}^{--} - x_{12}^{--}z$, \, 

$x_{13}^{+-}x_{23}^{++} - x_{12}^{++}z$, \, 
$x_{13}^{+-}x_{23}^{-+} - x_{12}^{+-}z$, \, 
$x_{13}^{--}x_{23}^{++} - x_{12}^{-+}z$, \, 
$x_{13}^{--}x_{23}^{-+} - x_{12}^{--}z$, \,

\item[(iv)]
$x_{12}^{++}x_{12}^{--} - z^{2}$, \, 
$x_{13}^{++}x_{13}^{--} - z^{2}$, \, 
$x_{23}^{++}x_{23}^{--} - z^{2}$, \, 

\item[(v)]
$x_{12}^{+-}x_{12}^{-+} - z^{2}$, \, 
$x_{13}^{+-}x_{23}^{-+} - z^{2}$, \, 
$x_{23}^{+-}x_{13}^{-+} - z^{2}$, \, 

\item[(vi)]
$x_{23}^{++}x_{23}^{+-} - x_{12}^{++}x_{12}^{-+}$, \,$x_{23}^{-+}x_{23}^{--} - x_{12}^{+-}x_{12}^{--}$, \,

$x_{23}^{++}x_{23}^{-+} - x_{13}^{++}x_{13}^{-+}$, \,
$x_{23}^{+-}x_{23}^{--} - x_{13}^{+-}x_{13}^{--}$, \,

\item[(vii)]
$x_{12}^{++}x_{12}^{+-} - x_{13}^{++}x_{13}^{+-}$, \ 
$x_{12}^{-+}x_{12}^{--} - x_{13}^{-+}x_{13}^{--}$
\end{enumerate} 
is the Gr\"obner basis 
$\Gc$ of $I_{{\bf D}_{n}^{\pm}}$ with respect to $\prec$.
\end{Example}

We have the following 
since the binomial 
$x_{1i}^{+p}x_{1i}^{-p}x_{jk}^{qr} - x_{ij}^{pq}x_{ik}^{pr}z$ (viii) where
$1 < i, j, k \leq n$ and $|\{i,j,k\}|=3$ satisfies
\begin{eqnarray*}
x_{1i}^{+p}x_{1i}^{-p}x_{jk}^{+r} - x_{ij}^{p+}x_{ik}^{pr}z
&=& 
x_{jk}^{+r} (x_{1i}^{+p}x_{1i}^{-p} - x_{ij}^{p+}x_{ij}^{p-})
+  x_{ij}^{p+} ( x_{ij}^{p-} x_{jk}^{+r} - x_{ik}^{pr}z ),
\\
x_{1i}^{+p}x_{1i}^{-p}x_{jk}^{-r} - x_{ij}^{p-}x_{ik}^{pr}z
&=&
x_{jk}^{-r} (x_{1i}^{+p}x_{1i}^{-p} - x_{ij}^{p+}x_{ij}^{p-})
+  x_{ij}^{p-} ( x_{ij}^{p+} x_{jk}^{-r} - x_{ik}^{pr}z ).
\end{eqnarray*}

\begin{Corollary}
The toric ideal $I_{{\bf D}_{n}^{\pm}}$ is generated by binomials
{\rm (i)} -- {\rm (vii)} in ${\mathcal G}$ in Theorem {\rm \ref{gb}}.
In particular, $I_{{\bf D}_{n}^{\pm}}$ is generated by
quadratic binomials.
\end{Corollary}

\section{Numerical example}
In this section, we perform our Markov chain procedure to 
an imaginary data set. Our data set is constructed from 
actual experimental data as follows.
In \cite{Jun-et-al-2003}, the Box-Behnken design is used to apply the
response surface method. The purpose of this experiment is to determine the
optimal processing condition of a pulsed UV-light system to inactivate the
fungal spores of Aspergillus niger in corn meal. 
The three factors are A: Treatment time (20, 60, 100 second), B: Distance
from the UV strobe (3, 8, 13 cm), and C: Voltage input (2000, 2900, 3800 V). 
The response is the reduction of the Aspergillus niger in the 
$\log_{10}$ scale. In \cite{Jun-et-al-2003}, the first-order and the 
second-order polynomial models for the response are considered.
See \cite{Jun-et-al-2003} for detail description
of the data analysis. 
Because our method is for discrete data, 
we use the rounded values of ($10$ times of) the 
responses in this experimental data
and treat them as realizations of discrete variables.  
Then we have an imaginary data set in Table 
\ref{tbl:Aspergillus-data}.
\begin{table}[htbp]
\caption{Box-Behnken design matrix for the three factors and the response 
(The responses are rounded values of Table 1 of \cite{Jun-et-al-2003}.
Fitted values are calculated under the model (\ref{eqn:first-order-model}))}
\label{tbl:Aspergillus-data}
\begin{center}
\begin{tabular}{rrrrr}\hline
\multicolumn{1}{c}{Time (s)} & 
\multicolumn{1}{c}{Distance (cm)} & 
\multicolumn{1}{c}{Voltage (V)} & 
\multicolumn{1}{c}{Response} & 
\multicolumn{1}{c}{Fitted values}\\ \hline
$20$ & $3$ & $2900$ & $4$    & $4.04$\\
$20$ & $13$ & $2900$ & $3$   & $3.59$\\
$100$ & $3$ & $2900$ & $33$  & $31.58$\\
$100$ & $13$ & $2900$ & $30$ & $28.02$\\
$20$ & $8$ & $2000$ & $2$    & $2.12$\\
$20$ & $8$ & $3800$ & $5$    & $6.84$\\
$100$ & $8$ & $2000$ & $14$  & $16.57$\\
$100$ & $8$ & $3800$ & $50$  & $53.42$\\
$60$ & $3$ & $2000$ & $7$    & $6.29$\\
$60$ & $3$ & $3800$ & $21$   & $20.29$\\
$60$ & $13$ & $2000$ & $5$   & $5.58$\\
$60$ & $13$ & $3800$ & $20$  & $18.01$\\
$60$ & $8$ & $2900$ & $13$   & $10.64$ \\ \hline
\end{tabular}
\end{center}
\end{table}
Because the responses in the original data in \cite{Jun-et-al-2003} are
continuous values, 
we cannot emphasize our computational results 
from the applied statistical view. The purpose of this numerical
experiment is only to check that our method works for some discrete data.
However, it can also be natural 
to consider the fitting of the log-linear model (\ref{eqn:log-linear-model})
to the response because the original response is reported in $\log_{10}$
scale. 

For the response data in Table \ref{tbl:Aspergillus-data}, we consider
the fitting of the first-order model (\ref{eqn:first-order-model})
based on the likelihood ratio statistics (\ref{eqn:likelihood-ratio}). 
The fitted values under the null model is calculated in the last column of 
Table \ref{tbl:Aspergillus-data}. The likelihood ratio is $2.36$ with 
$9$ degree of freedom. Therefore the asymptotic $p$ value is $0.98$ from
the asymptotic $\chi_9^2$ distribution. 
To evaluate the fitting of the first-order model (\ref{eqn:first-order-model}),
we perform the Markov chain Monte Carlo method. We use the Gr\"obner basis
given in Example \ref{ex:n=3-case} as a Markov basis. 
After $50000$ burn-in steps
from the observed data as the initial state, 
we derive $100000$ samples by Metropolis-Hasting algorithm. 
Among these samples, 
$98834$ samples have the larger likelihood ratio values than the observed
$2.36$. Therefore the conditional $p$ value is estimated as $0.99$, which
suggests the good fitting of the first-order model 
(\ref{eqn:first-order-model}).
Figure \ref{fig:Aspergillus-histgram} is the histogram of the sample
likelihood ratio statistics with the asymptotic $\chi_9^2$ distribution.
\begin{figure*}[htbp]
\begin{center}
\includegraphics[width=100mm]{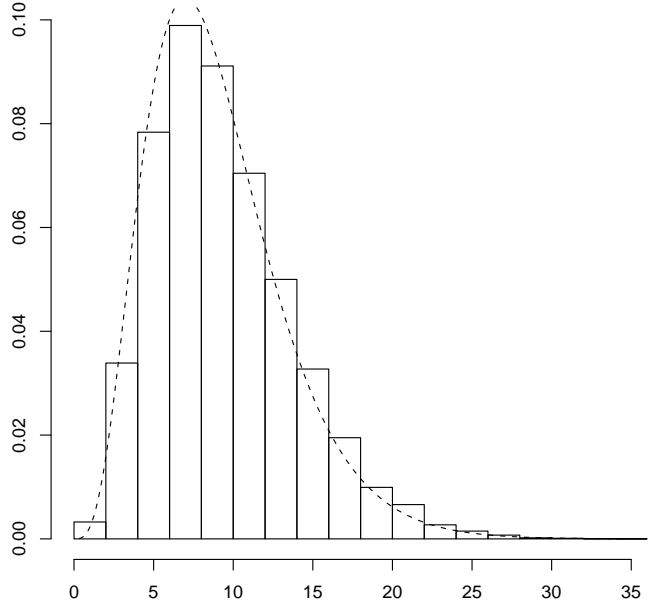}
\caption{Asymptotic and Monte Carlo estimated distribution of the likelihood
ratio statistics}
\label{fig:Aspergillus-histgram}
\end{center}
\end{figure*}

\section{Discussion}
In this paper, we present a new method for analyzing non-regular
fractional factorial designs. The motivation of this paper is a new finding
on the structure of the Gr\"obner bases of the centrally symmetric
configurations of root system $D_n$. 
As we have seen in the paper, 
we can relate the theoretical results in the algebraic
field to the statistical problems for the Box-Behnken designs. 
Our model is simple and fundamental. In fact, we usually consider more
complicated models such as second-order model for the analysis of the 
Box-Behnken designs. However, the structure of the Markov bases or 
the Gr\"obner bases for the second-order model 
is very complicated. Though the Markov chain Monte Carlo methods can be
considered for general non-regular designs, the structure of the Markov bases
is only revealed for simple models such as the hierarchical models for 
the regular designs at present. Therefore we think our contribution 
on the new results of the non-regular designs is important. 
Besides, compared to the continuous data analysis on the assumption of the 
normality, there are very few experiments are reported treating the 
discrete data arising in the fractional factorial designs. 
We think our Markov chain Monte Carlo procedure is very simple
and can be used easily, and can be one of the powerful choices
in the analysis of the discrete data.

\bibliographystyle{plain}

\end{document}